\documentclass[12pt]{amsart}
\usepackage{amsthm,amsxtra,amssymb,amscd}
\usepackage{mathtools}
\usepackage[T1]{fontenc}
\usepackage{lmodern}  % Latine Modern
\usepackage{stmaryrd}
\usepackage{bm}
\usepackage{accents}
\usepackage[numbers,sort&compress]{natbib}
\usepackage[all]{xy}
\usepackage[colorlinks]{hyperref}

\usepackage[top=25truemm,bottom=25truemm,left=30truemm,right=30truemm]{geometry}

\newtheorem{theorem}{Theorem}[section]
\newtheorem{corollary}[theorem]{Corollary}
\newtheorem{proposition}[theorem]{Proposition}
\newtheorem{lemma}[theorem]{Lemma}

\theoremstyle{definition}    
\newtheorem{definition}[theorem]{Definition}
\newtheorem{example}[theorem]{Example}
\newtheorem{remark}[theorem]{Remark}

\theoremstyle{remark}

\newcommand{\thmref}[1]{Theorem~\ref{#1}}
\newcommand{\secref}[1]{Section~\ref{#1}}
\newcommand{\lmmref}[1]{Lemma~\ref{#1}}
\newcommand{\prpref}[1]{Proposition~\ref{#1}}
\newcommand{\crlref}[1]{Corollary~\ref{#1}}
\newcommand{\rmkref}[1]{Remark~\ref{#1}}
\newcommand{\appref}[1]{Appendix~\ref{#1}}

\newcommand{\MB}{\mathcal{M}_{\mathrm{B}}}
\newcommand{\tMB}{\widetilde{\mathcal{M}}_{\mathrm{B}}}
\newcommand{\bSig}{\mathbf{\Sigma}}
\newcommand{\FS}[2]{{}_{#1} \calA_{#2}}
\newcommand{\st}{{\mathrm{st}}}
\newcommand{\bcdot}{\boldsymbol{\cdot}}

\newcommand{\relmiddle}[1]{\mathrel{}\middle#1\mathrel{}}
\newcommand{\set}[2]%
{\left\{ \, #1 \relmiddle| #2 \, \right\}}
\newcommand{\abs}[1]%
{\left\lvert #1 \right\rvert}
\newcommand{\ov}{\overline}
\newcommand{\wh}{\widehat}

\newcommand{\sumfrac}[2]{\genfrac{}{}{0pt}{2}{#1}{#2}}

\newcommand{\id}{\mathrm{id}} % identity
\newcommand{\Z}{\mathbb{Z}} % integers
\newcommand{\C}{\mathbb{C}} % complex numbers
\newcommand{\End}{\operatorname{End}} % endomorphisms
\newcommand{\Ker}{\operatorname{Ker}} % kernel
\newcommand{\prj}{\operatorname{pr}} % projection
\newcommand{\Lie}{\operatorname{Lie}} % Lie algebra
\newcommand{\Ad}{\operatorname{Ad}} % adjoint
\newcommand{\ad}{\operatorname{ad}} % adjoint
\newcommand{\spq}{/\!\!/} % symplectic quotient
\newcommand{\spqa}[1]{/\!\!/_{\!#1}\hspace{0.1em}}
\def\reE@DeclareMathSymbol#1#2#3#4{%
    \let#1=\undefined
    \DeclareMathSymbol{#1}{#2}{#3}{#4}}
\DeclareSymbolFont{symbolsC}{U}{txsyc}{m}{n}
\SetSymbolFont{symbolsC}{bold}{U}{txsyc}{bx}{n}
\DeclareFontSubstitution{U}{txsyc}{m}{n}
\reE@DeclareMathSymbol{\strictiff}{\mathrel}{symbolsC}{76}

\newcommand{\glue}[1]{\underset{#1}{\strictiff}}
\newcommand{\HHom}{\operatorname{\mathcal{H}\hspace{-.1em}\mathit{om}}} % Hom sheaf
\newcommand{\res}{\operatorname*{res}} % residue
\newcommand{\Sto}{\operatorname{\mathbb{S}to}}

\newcommand{\bbD}{\mathbb{D}}
\newcommand{\bbM}{\mathbb{M}}
\newcommand{\bbP}{\mathbb{P}}
\newcommand{\bfa}{\mathbf{a}}
\newcommand{\bfb}{\mathbf{b}}
\newcommand{\bfD}{\mathbf{D}}
\newcommand{\bfH}{\mathbf{H}}
\newcommand{\calA}{\mathcal{A}}
\newcommand{\calC}{\mathcal{C}}

\newcommand{\calO}{\mathcal{O}}
\newcommand{\calR}{\mathcal{R}}
\newcommand{\frkg}{\mathfrak{g}}
\newcommand{\frkh}{\mathfrak{h}}
\newcommand{\frkt}{\mathfrak{t}}
\newcommand{\frku}{\mathfrak{u}}
\newcommand{\frkC}{\mathfrak{C}}

\title[Unfolding of wild character varieties]{Unfolding of wild character varieties}

\author[K.~Hiroe]{Kazuki Hiroe}
\address{Department of Mathematics and Informatics, Chiba University, 
1-33, Yayoi-cho, Inage-ku, Chiba-shi, Chiba, 263-8522 JAPAN}
\email{kazuki@math.s.chiba-u.ac.jp}

\author[D.~Yamakawa]{Daisuke Yamakawa}
\address{Department of Mathematics, Faculty of Science Division I, Tokyo University of Science, 1-3
Kagurazaka, Shinjuku-ku, Tokyo, 162-8601 JAPAN}
\email{yamakawa@rs.tus.ac.jp}


\allowdisplaybreaks
\begin{document}
\mathtoolsset{showonlyrefs=true}
\maketitle

% \tableofcontents

\begin{abstract}
  In this paper, we study wild character varieties on compact Riemann surfaces 
  and construct Poisson maps from wild to tame character varieties by 
  unfolding irregular singularities into regular ones.
  Furthermore, we show that these unfolding Poisson maps induce Poisson 
  birational equivalences between wild and tame character varieties. 
  This result provides an affirmative answer to a conjecture posed by
  Klime\v{s}, Paul, and Ramis.
\end{abstract}

\section{Introduction}
This article investigates a geometric aspect of the confluence of singularities 
in meromorphic connections,
formulated through wild character varieties.
The main motivation is to construct Poisson and birational correspondences 
between wild character varieties and the character varieties arising 
from the unfolding of irregular singularities.

We now outline the main results of this article.
Let $G$ be a complex reductive group with a fixed maximal torus $T$.
Define $\bSig=(\Sigma,\bfa;{}^{1}Q,\ldots,{}^{m}Q)$
as an untwisted irregular curve in the sense of Boalch~\cite{Boa14}, 
consisting of a compact Riemann surface $\Sigma$ of genus $g$, 
a finite set of marked points $\bfa=\{a_{1},\ldots,a_{m}\}\subset \Sigma$,
and a collection of irregular types $^{i}Q$. 
Fixing a local coordinate $z_{i}$ centered at each marked point $a_{i}$, 
we write $^{i}Q=\sum_{j=1}^{r_{j}}{}^{i}Q_{j}z^{-j}_{i}
\in z_{i}^{-1}\mathfrak{t}[z_{i}^{-1}]$ and let 
\(\mathbf{H}=\prod_{i=1}^{m}Z_{G}(^{i}Q)
\)
denote the product of stabilizers of all coefficients 
${}^{i}Q_{j}\in \mathfrak{t}$ 
appearing in each $^{i}Q$. 
The wild character variety $\MB(\bSig)$
associated with $\bSig$ is then defined as a Poisson variety.
Furthermore,
given a conjugacy class $\mathcal{C}=\prod_{i=1}^m \mathcal{C}_{i} \subset \mathbf{H}$,
we define the symplectic wild character variety 
$\MB^{\st}(\bSig,\mathcal{C})$
as an open subset of 
the closed Poisson subvariety $\MB(\bSig,\ov{\mathcal{C}})$
of $\MB(\bSig)$,
corresponding to the closure $\ov{\mathcal{C}}$ of $\mathcal{C}$.
For precise definitions, see Section \ref{sec:wcv}.

Let us consider the unfolding of the wild character varieties.
Let $\bSig'=(\Sigma,\bfb;\mathbf{0})$ be 
the irregular curve with 
$\sum_{i=1}^{m}(r_{j}+1)=|\bfb|$ marked points, where all irregular types are  
trivial.
In other words, $\bSig'$ is simply a Riemann surface with marked points.
The associated wild character variety 
$\MB(\bSig')$ 
is then a standard (tame) character variety that parametrizes 
the isomorphism classes of semisimple $G$-representations of 
the fundamental group $\pi_{1}(\Sigma\backslash \bfb)$.
Next, we unfold the conjugacy class $\mathcal{C}$ as follows.
For each $i=1,\ldots,m$ and $j=1,\ldots,r_{i}$, choose $^{i}t_{j}\in T$ such that 
\[
  Z_{G}(^{i}t_{j})=Z_G({}^i Q_j, {}^i Q_{j+1}, \dots ,{}^i Q_{r_i}),\quad 
  Z_{G}(h({}^{i}t_{1}{}^{i}t_{2}\cdots {}^{i}t_{r_{i}})^{-1})\subset 
  Z_{G}(^{i}Q)\ (h\in \mathcal{C}_{i}).
\]
See \appref{app:parameter} for the existence of such elements.
Let  $^{i}\mathfrak{C}_{j}$ be the $G$-conjugacy class  
containing $^{i}t_{j}$, and 
let $^{i}\mathfrak{C}_{0}$ be the $G$-conjugacy class 
containing $\mathcal{C}_{i}({}^{i}t_{1}{}^{i}t_{2}\cdots {}^{i}t_{r_{i}})^{-1}$.
We then obtain the character varieties 
$\MB^{\st}(\bSig',\mathfrak{C})$
and $\MB(\bSig',\ov{\mathfrak{C}})$ 
associated with the conjugacy class $\mathfrak{C}=\prod_{i,j}{}^{i}\mathfrak{C}_{j}$.

The main results of this article are as follows. 
As a consequence of Theorem \ref{thm:unfolding-wcv}, 
we relate the wild character variety $\MB(\bSig,\ov{\mathcal{C}})$
and the unfolded character variety $\MB(\bSig',\ov{\mathfrak{C}})$
in a Poisson sense, namely, we obtain a Poisson map
\[
    \Upsilon_{({}^{i}t_{j})}\colon \MB(\bSig,\ov{\mathcal{C}})
    \rightarrow \MB(\bSig',\ov{\mathfrak{C}}).
\]
Furthermore, Corollary \ref{corl:unfolding-wcv} shows that 
the following Poisson birational equivalence holds:
\begin{theorem}\label{thm:main}
  Suppose that $\MB(\bSig,\ov{\calC})$, $\MB(\bSig',\ov{\frkC})$ 
  are both irreducible and 
  $\MB^\st(\bSig,\calC)$, $\MB^\st(\bSig',\frkC)$ are both nonempty. 
  Then $\MB(\bSig,\ov{\calC})$ and $\MB(\bSig',\ov{\frkC})$ 
  are Poisson birationally equivalent.
  Namely, there exist nonempty open subsets of them 
  which are isomorphic as Poisson varieties.
\end{theorem}

We now provide some background and related work.
In \cite{Hir24}, a deformation of moduli spaces of 
meromorphic connections on 
the trivial $G$-bundle over the Riemann sphere $\mathbb{P}^{1}$ 
was constructed via the unfolding of irregular singularities.
It was subsequently shown that any moduli space of meromorphic $G$-connections
with unramified (untwisted) irregular singularities admits 
a deformation to the moduli space of Fuchsian $G$-connections.
Our main theorem gives a generalization of this result 
in the context of wild character varieties and further  
establishes Poisson birational equivalences among the moduli spaces 
arising from such deformations.
In related work,  Klime\v{s} \cite{Kli} showed the existence of
a birational transformation between 
character varieties associated with the Painlev\'e V and Painlev\'e VI equations
in the context of nonlinear Stokes phenomena for Painlev\'e equations.
Building on this, Paul and Ramis \cite{PauRam} 
showed that the birational transformation by Klime\v{s} is 
in fact symplectic,
and they posed an open problem suggesting that  
character varieties for other types of Painlev\'e equations 
should admit similar symplectic and birational maps arising from the 
unfolding of their irregular singularities. 
Our main theorem provides an answer to this problem 
in the case of untwisted wild character varieties.

Throughout this article, we fix a complex reductive group $G$ 
together with an $\Ad$-invariant non-degenerate symmetric bilinear form 
$(\bcdot ,\bcdot)$ on its Lie algebra $\frkg = \Lie G$. 
A variety means a (possibly reducible) complex algebraic variety. 
The identity component of an algebraic group $H$ is denoted by $H^0$.

\subsection*{Acknowledgements}

K.H. was supported by JSPS KAKENHI Grant Number 25K07043.
D.Y. was supported by JSPS KAKENHI Grant Number 24K06695.

\section{Quasi-Hamiltonian geometry}

In this section we briefly recall some basic notions and facts 
in (algebraic) quasi-Hamiltonian geometry; 
for more details, see \cite{AKSM02,AMM98,BC05,Boa14,LBS15}.

\subsection{Quasi-Poisson and quasi-Hamiltonian structures}

For a (possibly singular) $G$-variety $M$ 
with tangent sheaf $\Theta_M = \HHom(\Omega^1_M,\calO_M)$, 
we denote by $\frkg \to \Gamma(M,\Theta_M)$, $\xi \mapsto \xi_M$ 
the corresponding infinitesimal action; 
for $\xi \in \frkg$, the vector field $\xi_M$ is characterized by  
\[
  \xi_M(df)(p) = \left. \frac{d}{dt} f(e^{-t\xi}\cdot p) \right|_{t=0} 
  \quad (p \in M,\ f \in \calO_{M,p}).
\]
For each $k \in \Z_{>0}$ it induces a map 
$(\bcdot)_M \colon \bigwedge^k \frkg \to \Gamma(M, \bigwedge^k \Theta_M)$. 
Define a homomorphism of $\calO_M$-modules 
$(\bcdot)_M \colon \calO_M \otimes_\C \bigwedge^k \frkg \to \bigwedge^k \Theta_M$ 
by $f \otimes \xi \mapsto f \xi_M$. 

Let $(\bcdot)^\vee \colon \frkg^* \to \frkg$, $\alpha \mapsto \alpha^\vee$ be the inverse of the isomorphism 
$\xi \mapsto (\xi,\bcdot)$.
Define $\frkg$-valued vector fields $\delta^L, \delta^R \in \Gamma(G,\Theta_G \otimes_\C \frkg)$ on $G$ by 
\[
  \delta^L_x (\theta) = (L_x^* \theta)^\vee, \quad \delta^R_x (\theta) = (R_x^* \theta)^\vee 
  \quad (x \in G,\ \theta \in T_x^* G).
\]
Put $\delta = \tfrac12 (\delta^L + \delta^R)$.

\begin{definition}
  (1) A \emph{quasi-Poisson $G$-variety} is a $G$-variety $M$ 
  equipped with a $G$-invariant skew-symmetric bracket operation 
  \[
    \{ \bcdot,\bcdot \} \colon \calO_M \otimes_\C \calO_M \to \calO_M,
  \]
  on the structure sheaf $\calO_M$ satisfying  
  \begin{align}
      \{ f,gh \} &= g\{ f,h \} + \{ f,g \} h, \\
      \{ f, \{ g, h \} \} + \{ g, \{ h, f \} \} + \{ h, \{ f, g \} \} 
      &= \chi_M(df,dg,dh) 
    \qquad (f,g,h \in \calO_M),
  \end{align}
  where $\chi \in \bigwedge^3 \frkg$ is defined by  
  \[
  \chi(\alpha,\beta,\gamma) = 
  \langle \alpha, [\beta^\vee,\gamma^\vee] \rangle \quad (\alpha,\beta,\gamma \in \frkg).
  \] 
  The bracket $\{ \bcdot,\bcdot \}$ is called a \emph{quasi-Poisson bracket}. 
  
  (2) For a quasi-Poisson $G$-variety $M$, 
  a $G$-equivariant morphism $\mu \colon M \to G$ 
  (where $G$ acts on itself by conjugation) 
  is called a \emph{moment map} if 
  the bivector field 
  $\Pi \in \Gamma(M, \bigwedge^2 \Theta_M)$ 
  corresponding to the bracket $\{ \bcdot,\bcdot \}$ satisfies
  \[
    \Pi(\mu^*\theta,\bcdot) = \left( \delta(\theta) \circ \mu \right)_M \quad (\theta \in \Omega^1_M).
  \]
  A quasi-Poisson $G$-variety equipped with a moment map  
  is called a \emph{Hamiltonian quasi-Poisson $G$-variety}.
\end{definition}

Note that if $G$ is abelian (e.g.\ $G= \{ 1 \}$), then $\chi =0$ and hence 
a quasi-Poisson structure is a usual Poisson structure.

Just as a non-degenerate Poisson structure comes from a symplectic structure, 
so some class of Hamiltonian quasi-Poisson structures come from 
the following geometric objects.

\begin{definition}
  A \emph{quasi-Hamiltonian $G$-space} is a smooth $G$-variety 
  equipped with a $G$-invariant two-form $\omega$ on $M$ 
  and a $G$-equivariant morphism $\mu \colon M \to G$ 
  (where $G$ acts on itself by conjugation) 
  satisfying the following conditions.
  \begin{enumerate}\setlength{\itemsep}{2pt}
    \item[(QH1)] $d\omega = \tfrac{1}{12}(\mu^{-1}d\mu,[\mu^{-1}d\mu,\mu^{-1}d\mu])$.
    \item[(QH2)] $\omega(\xi_M,\bcdot) = (\mu^{-1}d\mu+ d\mu\,\mu^{-1},\xi)$ for all $\xi \in \frkg$.
    \item[(QH3)] $\Ker \omega_p \cap \Ker(d\mu)_p = \{0\}$ for any $p \in M$.
  \end{enumerate}
  The two-form $\omega$ is called the \emph{quasi-Hamiltonian two-form} and 
  $\mu$ is called the \emph{moment map}.
\end{definition}

If $G$ is abelian, then the three axioms imply that $\omega$ is a symplectic form.

\begin{theorem}[{\cite[Theorem~10.3]{AKSM02}, \cite[Corollary~3.22]{BCS05}}]\label{thm:qp-qham}
  For any quasi-Hamiltonian $G$-space $(M,\omega,\mu)$,  
  there exists a unique quasi-Poisson $G$-structure 
  $\{ \bcdot,\bcdot \}$ on $M$ with $\mu$ a moment map 
  such that the corresponding bivector field 
  $\Pi \in \Gamma(M, \bigwedge^2 \Theta_M)$ satisfies
  \[
    \Pi(\iota_v \omega , \bcdot) = v - \frac14 \left( (\mu^{-1}d\mu - d\mu\,\mu^{-1})(v) \right)_M 
    \quad (v \in \Theta_M).
  \]
\end{theorem}

A Hamiltonian quasi-Poisson structure coming from 
a quasi-Hamiltonian structure in the above way 
is said to be \emph{non-degenerate}; 
see \cite{AKSM02} for an explicit characterization of non-degeneracy.

\begin{example}\label{eg:G-conj}
  The bracket $\{ \bcdot,\bcdot \} \colon \calO_G \otimes_\C \calO_G \to \calO_G$ defined by 
  \[
    \{ f,g \} = \frac12 (\delta^R(df), \delta^L(dg)) - \frac12 (\delta^R(dg), \delta^L(df)) 
    \quad (f,g \in \calO_G)
  \]
  is a quasi-Poisson $G$-structure on $G$ with respect to the conjugation action 
  (and does not depend on the choice of orthonormal basis) 
  and the identity map $\id \colon G \to G$ is a moment map.
  By the moment map condition, 
  the vector fields $\{ f, \bcdot \}$, $f \in \calO_G$ 
  are tangent to conjugacy classes. 
  It follows that any $G$-invariant subvariety $Z \subset G$ 
  is a quasi-Poisson $G$-subvariety of $G$, which means that the bracket $\{ \bcdot,\bcdot \}$ 
  descends to a quasi-Poisson bracket on $\calO_Z$. 
  Moreover the inclusion $Z \hookrightarrow G$ is a moment map for $Z$.
  In particular, any conjugacy class $\calC \subset G$ 
  is a Hamiltonian quasi-Poisson $G$-variety with moment map given by the inclusion. 
  In fact, the quasi-Poisson structure on $\calC$ comes from 
  the quasi-Hamiltonian two-form 
  \[
  \omega_x(\xi_G,\eta_G) = \frac12 (\xi,\Ad_x \eta) - \frac12 (\eta,\Ad_x \xi) \qquad (x \in \calC,\ \xi,\eta \in \frkg).
  \]
  If we fix $x \in \calC$, then the pullback of $\omega$ along 
  the map $\pi \colon G \to \calC$, $C \mapsto C^{-1}xC$ is given by 
  \[
  \pi^*\omega = \frac12 (dC\,C^{-1},\Ad_x(dC\,C^{-1})).
  \]
\end{example}

\begin{example}\label{eg:double}
  Let $G \times G$ act on $\bfD(G) = G \times G$ by 
  $(g,k) \cdot (C,h) = (k C g^{-1}, k h k^{-1})$. 
  Then $\bfD(G)$ is a quasi-Hamiltonian $G \times G$-space with moment map
  \[
  \mu \colon \bfD(G) \to G \times G, \quad (C,h) \mapsto (C^{-1}hC,h^{-1})
  \]
  and quasi-Hamiltonian two-form
  \[
  \omega = \frac12 (dC\,C^{-1}, \Ad_h (dC\,C^{-1})) 
  + \frac12 (dC\,C^{-1}, h^{-1}dh + dh\,h^{-1}).
  \]
\end{example}

\subsection{Reduction}

Let $H$ be another complex reductive group 
equipped with an $\Ad$-invariant non-degenerate symmetric bilinear form on 
$\frkh = \Lie H$. 

The following is the complex algebraic version of 
a particular case of \cite[Theorem~1.C]{LBS15} 
(the proof is similar).
Recall that a \emph{good quotient} of a $G$-variety $M$ 
in the sense of Seshadri~\cite{Ses72} 
is a variety $M/G$ together with 
a $G$-invariant surjective affine morphism $\pi \colon M \to M/G$ 
such that for any affine open subset $U \subset M/G$, 
the restriction $\pi^{-1}(U) \xrightarrow{\pi} U$ is an affine quotient: 
$\C[U]=\C[\pi^{-1}(U)]^G$.

\begin{proposition}\label{prp:reduction1}
  Let $M$ be a quasi-Poisson $G \times H$-variety. 

  $(1)$ Suppose that the $G$-action on $M$ has a good quotient $M/G$.
  Then the quasi-Poisson structure on $M$ descends to 
  a quasi-Poisson $H$-structure on the quotient $M/G$. 
  In particular, if $H=\{1\}$, then $M/G$ is a Poisson variety.

  $(2)$ Suppose that $M$ is equipped with a moment map 
  $\mu = (\mu_G,\mu_H) \colon M \to G \times H$  
  and let $Z \subset G$ be a $G$-invariant subvariety. 
  If the preimage $\mu_G^{-1}(Z)$ has a good quotient $\mu_G^{-1}(Z)/G$, 
  then the quasi-Poisson structure on $M$ 
  induces a quasi-Poisson $H$-structure on  
  $\mu_G^{-1}(Z)/G$ and 
  the morphism $\mu_G^{-1}(Z)/G \to H$ induced from $\mu_H$ 
  is a moment map.
\end{proposition}

We denote by $M \spqa{Z} G$ 
the above Hamiltonian quasi-Poisson $H$-variety $\mu_G^{-1}(Z)/G$ 
and call it the \emph{reduction} of $M$ by $G$ along $Z$. 
If $Z=\{1\}$, then we simply write $M \spq G = M \spqa{\{ 1 \}} G$.

When $M$ is quasi-Hamiltonian, the following also holds (see \cite[Theorem~5.1]{AMM98} and \cite[Proposition~10.6]{AKSM02}).

\begin{proposition}\label{reduction2}
  Let $M$ be a quasi-Hamiltonian $G \times H$-space 
  with moment map $\mu=(\mu_G,\mu_H) \colon M \to G \times H$
  and let $\calC \subset G$ be a conjugacy class. 
  Suppose that the $G$-action on $\mu_G^{-1}(\calC)$ is free and 
  has a geometric quotient in the sense of Seshadri, i.e., 
  it has a good quotient $\pi \colon \mu_G^{-1}(\calC) \to \mu_G^{-1}(\calC)/G$ 
  whose fibers are single orbits.  
  Then $\mu_G^{-1}(\calC)$, $\mu_G^{-1}(\calC)/G$ are non-singular  
  and the quasi-Hamiltonian two-form restricted to $\mu_G^{-1}(\calC)$ 
  descends to a two-form on $\mu_G^{-1}(\calC)/G$, 
  which gives the structure of a quasi-Hamiltonian $H$-space on 
  $\mu_G^{-1}(\calC)/G$ 
  with moment map induced from $\mu_H$ 
  $($in particular, $\mu_G^{-1}(\calC)/G$ is a smooth symplectic variety if $H=\{ 1\})$. 
  The corresponding quasi-Poisson structure on $\mu_G^{-1}(\calC)/G = M \spqa{\calC} G$ 
  coincides with the one given in the above proposition.
\end{proposition}

\begin{remark}
  If we allow the $G$-action on $\mu_G^{-1}(\calC)$ to have finite stabilizers in the above, 
  then $M \spqa{\calC} G$ may have quotient singularities, but still has 
  a quasi-Hamiltonian $H$-structure if we regard it as an orbifold (a smooth Deligne--Mumford stack).
\end{remark}

\begin{example}
  For any conjugacy class $\calC \subset G$, 
  the reduction $\bfD(G) \spqa{\calC^{-1}} G$ 
  of the double along the inverse conjugacy class $\calC^{-1}$ 
  by the action of the second $G$-factor is isomorphic to $\calC$.
\end{example}

\subsection{Fusion and gluing}

\begin{proposition}[{\cite[Propositions~5.1, 10.7]{AKSM02}}]
  Let $M$ be a Hamiltonian quasi-Poisson $G \times G \times H$-variety 
  with moment map $(\mu_1,\mu_2,\mu_H) \colon M \to G \times G \times H$. 
  Define $\frkg$-valued vector fields $\delta^i \in \Gamma(M,\Theta_M \otimes_\C \frkg)$, $i=1,2$ by 
  \[
    (\delta^i_M(\theta), \xi) = \xi^i_M(\theta) \quad (\theta \in \Omega^1_M),
  \]
  where $\xi \mapsto \xi^i_M$ is the infinitesimal action of the $i$-th $G$-factor of $G \times G \times H$.
  Then  
  \[
    \{ f,g \}_{\mathrm{fus}} \coloneqq \{ f,g \} 
    - \frac12 (\delta^1(df), \delta^2(dg)) + \frac12 (\delta^1(dg), \delta^2(df))  
    \quad (f,g \in \calO_M)    
  \]
  defines a quasi-Poisson structure on $M$ for the diagonal $G \times H$-action 
  with moment map $(\mu_1 \cdot \mu_2,\mu_H) \colon M \to G \times H$. 

  If $M$ is smooth and non-degenerate with quasi-Hamiltonian two-form $\omega$, 
  then so is $(M,\{ \bcdot,\bcdot \}_{\mathrm{fus}})$ 
  and the corresponding quasi-Hamiltonian two-form is given by
  \[
    \omega_{\mathrm{fus}} = \omega - \frac12 (\mu_1^{-1}d\mu_1, d\mu_2\, \mu_2^{-1}).
  \]
\end{proposition}

We call this procedure the (internal) \emph{fusion}.
For instance, if $M_i$ is a Hamiltonian quasi-Poisson $G \times H_i$-variety for $i=1,2$, 
then the product $M_1 \times M_2$ is 
a Hamiltonian quasi-Poisson $G \times H_1 \times G \times H_2$-variety, 
from which we obtain a Hamiltonian quasi-Poisson $G \times H_1 \times H_2$-variety
by fusing the two $G$-factors.
We denote this Hamiltonian quasi-Poisson $G \times H_1 \times H_2$-variety by 
$M_1 \circledast_G M_2$ (or simply $M_1 \circledast M_2$) 
and call it the \emph{fusion product} of $M_1$ and $M_2$.
If we can perform the reduction 
$(M_1 \circledast M_2) \spq G$, 
then we call it the \emph{gluing} of $M_1$ and $M_2$
and denote it by $M_1 \glue{G} M_2$.

\section{Fission spaces and their gluing}\label{sec:fission}

In \cite{Boa09,Boa14}, Boalch introduced a new class of quasi-Hamiltonian spaces, 
called (higher) \emph{fission spaces}, 
and used them to construct Poisson structures on (untwisted) wild character varieties.
From now on, we fix a maximal torus $T \subset G$ with Lie algebra $\frkt \subset \frkg$.

Let $P$ be a parabolic subgroup of $G$ containing the maximal torus $T$ 
and $H$ be a unique Levi subgroup of $P$ containing $T$. 
Let $U^+$ be the unipotent radical of $P$ and $U^-$ be that of the opposite parabolic subgroup. 
For a positive integer $r$, put 
\[
  \FS{G}{H}^r = G \times H \times (U^+ \times U^-)^r,
\]
which we call a fission space. 
We denote an element of $(U^+ \times U^-)^r$ by 
$(u_1,u_2, \dots ,u_{2r})$, where $u_{\mathrm{odd}} \in U^+$ and $u_{\mathrm{even}} \in U^-$. 
Let $G \times H$ act on $\FS{G}{H}^r$ by 
\[
  (g,k) \cdot (C,h,u_1, \dots ,u_{2r}) = 
  (kCg^{-1},khk^{-1},ku_1 k^{-1}, \dots ,k u_{2r} k^{-1}).
\]
The bilinear form $(\bcdot ,\bcdot)$ on $\frkg$ restricts to 
a non-degenerate $\Ad$-invariant symmetric bilinear form on $\frkh = \Lie H$.

\begin{theorem}[{\cite[Theorem~3.1]{Boa14}}]
  The $G \times H$-variety $\FS{G}{H}^r$ is a quasi-Hamiltonian $G \times H$-space 
  with moment map 
  \[
    \mu \colon (C,h,u_1, \dots ,u_{2r}) \mapsto (C^{-1} h u_1 u_2 \cdots u_{2r} C, h^{-1}),
  \]
  and quasi-Hamiltonian two-form defined by 
    \begin{equation}\label{eq:form}
      \begin{split}
    2\omega &= (dC\,C^{-1},\Ad_b(dC\,C^{-1})) + (dC\,C^{-1},db\,b^{-1})   \\ 
    &\quad + (dC_0\,C_0^{-1},h^{-1}dh) - \sum_{j=0}^{2r-1} (C_j^{-1}dC_j,C_{j+1}^{-1}dC_{j+1}),
      \end{split}
  \end{equation}
  where $C_j = u_{j+1} u_{j+2} \cdots u_{2r} C$ (so that $C_{2r}=C$), $b=h u_1 u_2 \cdots u_{2r}$.  
\end{theorem}

More generally, let 
\[
  P_1 \subset P_2 \subset \cdots \subset P_r
\]
be an increasing sequence of parabolic subgroups of $G$ 
containing $T$.
For $j=1,2, \dots ,r$, let $U_j^+$, $U_j^-$, $H_j$ 
be the unipotent radical of $P_j$, 
that of the opposite parabolic subgroup, 
the Levi subgroup of $P_j$ containing $T$, respectively.
Then $H_1 \subset H_2 \subset \cdots \subset H_r$ 
and $U_1^\pm \supset U_2^\pm \supset \cdots \supset U_r^\pm$.
Put 
\[
  \calA = G \times H_1 \times \prod_{j=1}^r (U_j^+ \times U_j^-),
\]
which we call a \emph{multi-fission space}. 
Note that this is 
a $G \times H_1$-invariant closed subvariety 
of a fission space $\FS{G}{H_1}^r$, 
and equal to $\FS{G}{H_1}^r$ if $P_1 = P_2 = \cdots = P_r$. 

\begin{proposition}\label{prp:multi}
  The multi-fission space $\calA \subset \FS{G}{H_1}^r$ 
  is a quasi-Hamiltonian $G \times H_1$-space, 
  where the moment map and the quasi-Hamiltonian two-form 
  are the restrictions of those for $\FS{G}{H_1}^r$.
\end{proposition}

\begin{proof}
  For each $j=1,2, \dots ,r$, the intersection $P_j \cap H_{j+1}$ 
  is a parabolic subgroup of $H_{j+1}$ with unipotent radical $V_j^+ \coloneqq U_j^+ \cap H_{j+1}$ 
  and $V_j^- \coloneqq U_j^- \cap H_{j+1}$ is that of the opposite parabolic subgroup. 
  Furthermore, the product map 
  \[
    V_j^\pm \times V_{j+1}^\pm \times \cdots \times V_r^\pm \to U_j^\pm
  \] 
  is an isomorphism of varieties; in other words, 
  $V_j^\pm, V_{j+1}^\pm, \dots ,V_r^\pm$ directly span $U_j^\pm$ in the sense of Borel~\cite{Bor}. 
  Thus \cite[Theorem~6.4]{Boa14} shows that there exists 
  a $G \times H_1$-equivariant isomorphism of varieties 
  \[
    \calA \simeq \calA(1) \glue{H_2} \calA(2) \glue{H_3} \cdots \glue{H_r} \calA(r), 
  \]
  where $\calA(j)$ is the fission space
  \[
    \calA(j) \coloneqq H_{j+1} \times H_j \times (V_j^+ \times V_j^-)^j = \FS{H_{j+1}}{H_j}^j, 
  \]
  along which the pull-back of the moment map and the quasi-Hamiltonian two-form 
  are exactly the restrictions of those for $\FS{G}{H_1}$. 
\end{proof}

\section{Wild character varieties}\label{sec:wcv}

In this section we briefly recall the (untwisted) wild character varieties following \cite{Boa14}. 

First, we introduce a building piece of wild character varieties. 
Let $z$ be a local coordinate centered at the marked point $z=0$ 
on a pointed Riemann surface.
An (untwisted) \emph{irregular type} at $z=0$ is an element of 
$z^{-1}\frkt[z^{-1}]$. 
Take an irregular type 
$Q(z)=\sum_{j=1}^r Q_j z^{-j} \in z^{-1}\frkt[z^{-1}]$, $Q_r \neq 0$. 
For a root $\alpha \in \frkt^*$ of $G$ relative to $T$, put
\[
  q_\alpha(z) = \langle \alpha, Q \rangle 
  = \sum_{j=1}^r \langle \alpha, Q_j \rangle z^{-j} 
  \in z^{-1}\C[z^{-1}].
\]
A direction $d \in [0,2\pi) \simeq S^1$ is called a \emph{singular direction}  
supported by a root $\alpha$ if 
$\exp(q_\alpha(z))$ has maximal decay as $z \to 0$ in the direction $d$, 
i.e., if $\langle \alpha, Q_r \rangle e^{-\sqrt{-1}d r}$ is a negative real number.
Let $s$ be the number of singular directions at for all roots and label them as 
\[
  0 \leq d_1 < d_2 < \cdots < d_s < 2\pi.
\]
It is known that for each singular direction $d_i$,
the sum of the root spaces $\frkg_\alpha \subset \frkg$ 
for all roots $\alpha$ supporting $d_i$
is a nilpotent Lie subalgebra of $\frkg$. 
The subgroup $\Sto_i(Q) \subset G$ obtained by exponentiating this Lie subalgebra 
is called the \emph{Stokes group} at $d_i$.
Let $H \equiv Z_G(Q) = Z_G(Q_1,Q_2, \dots ,Q_r)$ be the common centralizer 
of the coefficients of $Q$ in $G$. 
Observe that the Stokes groups $\Sto_i(Q)$ are normalized by $H$.

\begin{theorem}[{\cite[Theorem~7.6]{Boa14}}]
  The $G \times H$-variety 
  \[
    \calA(Q) \coloneqq G \times H \times \prod_{i=1}^s \Sto_i(Q)
  \]
  is a quasi-Hamiltonian $G \times H$-space with moment map
  \[
    \mu \colon \calA(Q) \to G \times H, \quad (C,h,S_1, \dots ,S_s) 
    \mapsto (C^{-1} h S_s \cdots S_2 S_1 C, h^{-1})
  \]
  and quasi-Hamiltonian two-form $\omega$ defined by
  \begin{align}
    2\omega &= (dC\,C^{-1},\Ad_b(dC\,C^{-1})) + (dC\,C^{-1},db\,b^{-1})   \\ 
    &\quad + (dC_s\,C_s^{-1},h^{-1}dh) - \sum_{i=1}^s (C_i^{-1}dC_i,C_{i-1}^{-1}dC_{i-1}),
  \end{align}
  where $C_i = S_i \cdots S_2 S_1 C$ (so that $C_0=C$), $b=hS_s \cdots S_2 S_1$.  
\end{theorem}

We can describe the space $\calA(Q)$ more concretely. 
Define an increasing sequence of reductive subgroups of $G$ by
\[
  H = H_1 \subset H_2 \subset \cdots \subset H_r \subset H_{r+1} = G, \quad 
  H_j = Z_G(Q_j,Q_{j+1}, \dots ,Q_r).
\]

\begin{proposition}\label{prp:triangular}
  There exist an increasing sequence of parabolic subgroups 
  \[
    P_1 \subset P_2 \subset \cdots \subset P_r
  \]  
  of $G$ wth each $P_i$ containing $H_i$ as a Levi subgroup 
  such that $\calA(Q)$ is isomorphic to the associated multi-fission space 
  as a quasi-Hamiltonian $G \times H$-space.
\end{proposition}

\begin{proof}
  By \cite[Proposition~7.12]{Boa14}, 
  there exists a parabolic subgroup $P'_j$ of $H_{j+1}$ with Levi subgroup $H_j$ 
  for each $j=1,2, \dots ,r$ such that if we denote by $V_j^+$ 
  the unipotent radical of $P'_j$ and by $V_j^-$ that of the opposite parabolic subgroup, 
  then 
  \[
    \calA(Q) \simeq \calA(1) \glue{H_2} \calA(2) \glue{H_3} \cdots \glue{H_r} \calA(r), 
  \]
  where $\calA(j) = H_{j+1} \times H_j \times (V_j^+ \times V_j^-)^j = \FS{H_{j+1}}{H_j}^j$.
  By the proof of \prpref{prp:multi}, the right hand side is isomorphic to the multi-fission space.
\end{proof}

Now let $\Sigma$ be a compact Riemann surface of genus $g$ and take a finite set 
$\bfa = \{a_1, \ldots, a_m\} \subset \Sigma$ of marked points.
Take a local coordinate $z_i$ centered at $a_i$ and  
an irregular type ${}^i Q = \sum_{j=1}^{r_i} {}^i Q_j z_i^{-j}$ 
(of pole order $r_i$) at $a_i$ for each $i=1,2,\ldots,m$.
The tuple $\bSig \coloneqq (\Sigma,\bfa;{}^1 Q, \dots ,{}^m Q)$ 
is called an (untwisted) \emph{irregular curve} with structure group $G$.

Put 
\[
  \bfH = Z_G({}^1 Q) \times \cdots \times Z_G({}^m Q),
\]
and 
\[
  \calR(\bSig) \coloneqq \left( \underbrace{\bbD(G) \circledast_G \cdots \circledast_G \bbD(G)}_{g} 
  \circledast_G \calA({}^1 Q) \circledast_G \cdots \circledast_G \calA({}^m Q) \right) \spq G,
\]
which is a quasi-Hamiltonian $\bfH$-space.

\begin{definition}
  The affine quotient $\MB(\bSig) \coloneqq \calR(\bSig)/\bfH$ 
  is called the \emph{wild character variety} $\MB(\bSig)$ 
  associated to the irregular curve $\bSig$. 
\end{definition}

It is known that 
the quasi-Hamiltonian structure on $\calR(\bSig)$
induces a Poisson structure on $\MB(\bSig)$; 
see \cite[Proposition~2.8]{Boa14}.

We are interested in some subvarieties of $\MB(\bSig)$. 
Let $\mu_{\bfH} \colon \calR(\bSig) \to \bfH$ be the moment map
\[
  \left[ \left( (A_l,B_l)_{l=1}^g, (C_i,h_i,{}^i S_1, \dots ,{}^i S_{s_i})_{i=1}^m \right) \right] 
  \mapsto (h_1^{-1}, \dots ,h_m^{-1}),
\]
where $s_i$ is the number of singular directions at $a_i$. 
This map takes values in the following subgroup of $\bfH$.

\begin{proposition}\label{prp:value}
  Put $\ov{Z}(G) = G/[G,G] = Z(G)^0/(Z(G)^0 \cap [G,G])$ 
  and let $\prj_G \colon G \to \ov{Z}(G)$ be the canonical projection. 
  Then the image of $\mu_{\bfH}$ is contained in 
  the kernel of the homomorphism 
  \[
    \prj_{\bfH} \colon \bfH \to \ov{Z}(G), \quad (h_i) \mapsto \prod_{i=1}^m \prj(h_i).
  \]
\end{proposition}

\begin{proof}
  This is implicitly shown in the proof of \cite[Corollary~9.7]{Boa14}. 
  Any point $p=[(A_l,B_l), (C_i,h_i,({}^i S_j))] \in \calR(\bSig)$  
  satisfies the moment map relation   
  \[
    \prod_{l=1}^g [A_l,B_l] \prod_{i=1}^m C_i^{-1} h_i ({}^i S_{s_i} \cdots {}^i S_1) C_i = 1.
  \] 
  Since $\ov{Z}(G)$ is abelian and any unipotent element of $G$ lies in $\Ker \prj_G$, 
  applying $\prj_G$ to the both sides of the above relation yields
  $\prod_{i=1}^m \prj_G(h_i) = 1$.
\end{proof}

Put $\bfH' = \Ker \prj_{\bfH} \subset \bfH$.
Observe that the center $Z(G)$ embedded diagonally into $\bfH$ 
acts trivially on $\calR(\bSig)$ and the projection   
$\bfH' \to \bfH/Z(G)$ is an isogeny (i.e., is surjective with finite kernel).
Also, it is easy to see that the Lie algebra $\Lie \bfH'$ is perpendicular to 
$\Lie Z(G) \subset \Lie \bfH$ with respect to 
the invariant bilinear form on $\Lie \bfH$.  
Hence the restriction of the bilinear form to $\Lie \bfH'$ is non-degenerate. 

The quasi-Hamiltonian two-form on $\calR(\bSig)$ and 
the moment map $\mu_{\bfH} \colon \calR(\bSig) \to \bfH' \subset \bfH$ 
still satisfy axioms (QH1), (QH2), (QH3) 
for the action of the subgroup $\bfH' \subset \bfH$ 
equipped with the above bilinear form. 
Hence they make $\calR(\bSig)$ into a quasi-Hamiltonian $\bfH'$-space.

Let $\calC \subset \bfH'$ be a conjugacy class (note that 
it is also a conjugacy class of the group $\bfH = \bfH' Z(G)$). 
The affine quotient 
\[
  \MB(\bSig,\ov{\calC}) \coloneqq 
  \mu_{\bfH}^{-1}\left( \ov{\calC}^{-1} \right)/\bfH = \calR(\bSig) \spqa{\ov{\calC}^{-1}} \bfH',
\]
where $\ov{\calC}^{-1} = \{\, h^{-1} \mid h \in \ov{\calC}\,\}$,
is a closed Poisson subvariety of $\MB(\bSig)$. 
We say that a point $p \in \calR(\bSig)$ is \emph{stable} 
if its $\bfH$-orbit is closed and of dimension equal to $\dim \bfH / \dim Z(G) = \dim \bfH'$.
By geometric invariant theory, the stable locus 
(the set of stable points)  
$\mu_{\bfH}^{-1}\left( \calC^{-1} \right)^\st$ 
of $\mu_{\bfH}^{-1}\left( \calC^{-1} \right)$ 
has a geometric quotient
\[
  \MB^\st(\bSig,\calC) \coloneqq \mu_{\bfH}^{-1}\left( \calC^{-1} \right)^\st/\bfH 
  = \calR(\bSig)^\st \spqa{\calC^{-1}} \bfH',
\]
which is an open subset of $\MB(\bSig,\ov{\calC})$ 
and has the structure of a symplectic orbifold.
We also call $\MB(\bSig,\ov{\calC})$, $\MB^\st(\bSig,\calC)$  
wild character varieties.

By swapping the order of reductions, 
we can also describe $\MB(\bSig,\ov{\calC})$ and $\MB^\st(\bSig,\calC)$ as follows.
Note that $\calC$ has the form $\calC = \prod_{i=1}^m \calC_i$, where 
$\calC_i \subset Z_G({}^i Q)$ is a conjugacy class of $Z_G({}^i Q)$.
For $i=1,2, \dots ,m$, perform the reduction 
of $\calA({}^i Q)$ by $Z_G({}^i Q)$ along the inverse conjugacy class $\calC_i^{-1}$: 
\begin{align}
  \calA_{\calC_i}({}^i Q) 
  &\coloneqq \calA({}^i Q) \spqa{\calC_i^{-1}} Z_G({}^i Q) \\
  &= \set{(C_i,h_i,({}^i S_j)) \in \calA({}^i Q)}{h_i \in \calC_i}/Z_G({}^i Q).
\end{align}
Then the moment map for the fusion product 
\[
  \tMB(\bSig,\calC) \coloneqq \bbD(G)^{\circledast g} \circledast_G \calA_{\calC_1}({}^1 Q) \circledast_G \cdots \circledast_G \calA_{\calC_m}({}^m Q)
\]
takes values in the subgroup $G' \coloneqq \Ker \prj_G$, 
and the space $\MB^\st(\bSig,\calC)$ is the reduction of 
the stable locus of $\tMB(\bSig,\calC)$ by $G'$: 
\[
  \MB^\st(\bSig,\calC) = \tMB(\bSig,\calC)^\st \spq G',
\]
where a point $p \in \tMB(\bSig,\calC)$ is stable if its $G$-orbit is closed 
and of dimension equal to $\dim G / \dim Z(G) = \dim G'$ (see \rmkref{rmk:stable1} below). 
To obtain a similar description of $\MB(\bSig,\ov{\calC})$, put 
\begin{align}
  \calA_{\ov{\calC}_i}({}^i Q) 
  &\coloneqq \calA({}^i Q) \spqa{\ov{\calC_i}^{-1}} Z_G({}^i Q) \\
  &= \set{(C_i,h_i,({}^i S_j)) \in \calA({}^i Q)}{h_i \in \ov{\calC}_i}/Z_G({}^i Q).
\end{align}
This is a Hamiltonian quasi-Poisson $G$-variety 
The variety $\MB(\bSig,\ov{\calC})$ is described as   
\[
  \MB(\bSig,\ov{\calC}) = \tMB(\bSig,\ov{\calC}) \spq G',
\] 
where 
\[
  \tMB(\bSig,\ov{\calC}) \coloneqq \bbD(G)^{\circledast g} \circledast_G \calA_{\ov{\calC}_1}({}^1 Q) \circledast_G \cdots \circledast_G \calA_{\ov{\calC}_m}({}^m Q).
\]

\begin{remark}\label{rmk:stable1}
  Suppose that $p \in \calR(\bSig)$ and $q \in \tMB(\bSig,\calC)$ are 
  represented by the same point in $\bbD(G)^{\circledast g} \circledast_G \calA({}^1 Q) \circledast_G \cdots \circledast_G \calA({}^m Q)$.
  Then $p$ is stable (for the $\bfH$-action) if and only if 
  $q$ is stable (for the $G$-action).
  This follows from \cite[Proof of Theorem~19, Lemma~21]{BY15}. 
\end{remark}

\begin{remark}\label{rmk:stable2}
  For $i=1,2, \dots ,m$, let $Z_i$ be the identity component of the center of $Z_G({}^i Q)$.
  Then \cite[Theorem~9.3]{Boa14} together with the above remark implies that 
  a point $((A_l,B_l), ([C_i,h_i,({}^i S_j)])) \in \tMB(\bSig,\calC)$ 
  is stable if and only if there exists no proper parabolic subgroup of $G$ 
  containing all $A_l, B_l$, $C_i^{-1}h_i C_i$, 
  $C_i^{-1} {}^i S_j C_i$, $C_i^{-1}Z_i C_i$.
\end{remark}

\begin{remark}
  Our definition of wild character varieties 
  depends on the choice of local coordinates around marked points 
  and generators of the fundamental groupoid of some auxiliary surface. 
  The original definition does not depend on them; see \cite{Boa14,BY15}.
\end{remark}

\section{Triangular decomposition of conjugacy classes}\label{sec:triangular}

In this section we introduce 
a sort of ``triangular decomposition'' of conjugacy classes in $G$, 
which gives affine charts of conjugacy classes and will be used in the subsequent sections.

Let $P$ be a parabolic subgroup of $G$ 
and $H$ be a Levi subgroup of $P$.
Let $U^+$ be the unipotent radical of $P$ with Lie algebra $\frku$ 
and $U^-$ be that of the opposite parabolic subgroup.
  
\begin{lemma}\label{lmm:conj-unip}
  Take $h \in H$ so that the identity component $Z_G(h)^0$ 
  of the centralizer in $G$ is contained in $H$.
  Then the map $U^+ \to U^+$, $u \mapsto (h^{-1}u^{-1}h)u$ is an isomorphism of varieties.
  In particular, for any $u' \in U^+$, 
  there exists a unique $u \in U^+$ such that $hu' = u^{-1} h u$. 
\end{lemma}

\begin{proof}
  Let $\frku^i$, $i=0,1,\dots $ be the lower central series of 
  the Lie algebra $\frku \coloneqq \Lie U^+$: 
  \[
    \frku^0 = \frku, \quad \frku^i = [\frku, \frku^{i-1}] \quad (i>0).
  \]
  For $i=1,2, \dots$, take a vector subspace $\frku_i \subset \frku^{i-1}$ 
  complementary to $\frku^i$: 
  \[
    \frku^{i-1} = \frku_i \oplus \frku^i \quad (i>0).
  \] 
  Put $\frku'_i = (\Ad_h^{-1} -\id)(\frku_i)$. Note that $\Ad_h^{-1} -\id \colon \frku \to \frku$ 
  is a linear isomorphism preserving the filtration $\{ \frku^i \}_{i>0}$ 
  by the assumption for $h$.
  Since $\frku$ is nilpotent, we have $\frku = \bigoplus_{i>0} \frku_i = \bigoplus_{i>0} \frku'_i$. 
  For $X \in \frku$, let $Y(X) \in \frku$ be a unique element such that 
  \[
    e^{Y(X)} = (h^{-1}e^{-X}h)e^X = e^{-\Ad_h^{-1}(X)} e^X.
  \]
  By the Baker--Campbell--Hausdorff formula, $Y(X)$ is expressed as $Y(X) = \sum_{i>0} Y_i(X)$, where 
  $Y_1(X) = -\Ad_h^{-1}(X) + X$ and each $Y_i(X)$ is a linear combination of elements of the form 
  \[
    \ad_{Z_1} \ad_{Z_2} \cdots \ad_{Z_{i-1}} (Z_i),  \quad 
    Z_1, \dots ,Z_i \in \{ -\Ad_h^{-1}(X), X \}.
  \] 
  Observe that if we decompose $X$ as $X=\sum X_i$, $X_i \in \frku_i$, then 
  \begin{align}
    Y(X) + \frku^i &= Y_1(X) + Y_2(X) + \cdots + Y_i(X) + \frku^i \\ 
    &=-(\Ad_h^{-1} -\id)(X_i) + \sum_{j=1}^i Y_j(X_1 + X_2 + \cdots + X_{i-1}) + \frku^i.
  \end{align}
  Thus for any $Y=\sum Y_i \in \frku$, $Y_i \in \frku'_i$, 
  the equation $Y(X)=Y$ determines $X_i \in \frku_i$, $i=1,2, \dots $ inductively with 
  $X_1 = -(\Ad_h^{-1} -\id)^{-1}(Y_1)$, and $e^Y \mapsto e^X$ gives an inverse of the map $u \mapsto h^{-1}u^{-1}hu$.
\end{proof}

Take $h_0 \in H$ so that $Z_G(h_0)^0 \subset H$, 
and let $\calC \subset H$ (resp.\ $\frkC \subset G$) be its $H$-conjugacy class 
(resp.\ $G$-conjugacy class).
By the above lemma, we can define the following map: 
\[
  \tau \colon \calC \times U^+ \times U^- \to \frkC, \quad 
  (h,u,v) \mapsto v^{-1}(hu)v.
\]

\begin{proposition}\label{prp:conj-tri}
  The morphism $\tau$ is \'etale,  
  and is an open immersion if $Z_G(h_0) \subset H$. 
  Furthermore, if we denote by $\omega_H, \omega_G$ 
  the quasi-Hamiltonian two-forms on $\calC$, $\frkC$, respectively, then 
  \[
    2\tau^*\omega_G = 2\omega_H 
    + (dv\,v^{-1},(hu)^{-1}d(hu) + d(hu)(hu)^{-1}+ \Ad_{hu}(dv\, v^{-1})).
  \]
\end{proposition}

\begin{proof}
  The open immersion  
  \[
    H \times U^+ \times U^- \to G, \quad (k,u,v) \mapsto kuv
  \]
  induces an open immersion 
  \[
    \iota \colon \calC \times U^+ \times U^- \simeq H/Z_H(h_0) \times U^+ \times U^- \to G/Z_H(h_0).
  \]
  If $Z_H(h_0)=Z_G(h_0)$, i.e., $Z_G(h_0) \subset H$, 
  then $\iota$ may be identified with $\tau$ 
  through the isomorphism given in the previous proposition. 
  In general, $\tau$ is identified with the composite of $\iota$ 
  and the quotient map $G/Z_H(h_0) \to G/Z_G(h_0) \simeq \frkC$, 
  which is an \'etale morphism since $Z_H(h_0)$ contains the identity component $Z_G(h_0)^0$ by the assumption.
  Thus $\tau$ is also \'etale. 
  For $(k,w,v) \in H \times U^+ \times U^-$, 
  define $h \in \calC$, $C \in G$ and $u \in U^+$ by
  \[
    h = k^{-1}h_0 k, \quad C = kwv, \quad w^{-1}hw = hu.
  \]
  so that $C^{-1}h_0 C = v^{-1}w^{-1}h_0 wv = v^{-1}(hu)v$. 
  Also put $p=kw$.
  Then 
  \[
    dC\,C^{-1} = dp\,p^{-1} + \Ad_p(dv\,v^{-1}),
  \]
  and hence the pullback of $2\omega_G$ along the map 
  $\pi_G \colon (k,w,v) \mapsto C^{-1}h_0 C$ is given by 
  \begin{align}
    2\pi_G^*\omega_G 
    &= (dC\,C^{-1},\Ad_{h_0}(dC\,C^{-1})) \\
    &= (dp\,p^{-1} + \Ad_p(dv\,v^{-1}), \Ad_{h_0}(dp\,p^{-1} + \Ad_p(dv\,v^{-1})) \\
    &= (dp\,p^{-1},\Ad_{h_0}(dp\,p^{-1})) + (dp\,p^{-1},\Ad_{h_0 p}(dv\,v^{-1})) \\
    &\quad + (\Ad_p(dv\,v^{-1}), \Ad_{h_0} (dp\,p^{-1})) + (\Ad_p(dv\,v^{-1}), \Ad_{h_0 p}(dv\,v^{-1})). 
  \end{align}
  Since $p^{-1}h_0 p = hu$, we obtain  
  \begin{align}
    2\pi_G^*\omega_G 
    &= (dp\,p^{-1},\Ad_{h_0}(dp\,p^{-1})) + (p^{-1}dp,\Ad_{hu}(dv\,v^{-1})) \\
    &\quad + (dv\,v^{-1}, \Ad_{hu} (p^{-1}dp)) + (dv\,v^{-1}, \Ad_{hu}(dv\,v^{-1})) \\
    &= (dp\,p^{-1},\Ad_{h_0}(dp\,p^{-1})) \\
    &\quad + (dv\,v^{-1}, -\Ad_{hu}^{-1}(p^{-1}dp) + \Ad_{hu} (p^{-1}dp) + \Ad_{hu}(dv\,v^{-1})). 
  \end{align}
  Since $(\frku,\frkh+\frku)=0$, the first term $(dp\,p^{-1},\Ad_{h_0}(dp\,p^{-1}))$ 
  is equal to $(dk\,k^{-1},\Ad_{h_0}(dk\,k^{-1}))$, which is the pullback of 
  $2\omega_H$ along the map $\pi_H \colon (k,w,v) \mapsto k^{-1}h_0 k$. 
  On the other hand, we have $hu = p^{-1}h_0 p$, and hence 
  \begin{align}
    d(hu)(hu)^{-1} &= -p^{-1}dp + \Ad_{p^{-1}h_0}(dp\,p^{-1}) \\
    &= -p^{-1}dp + \Ad_{hu}(p^{-1}dp).
  \end{align}
  Similarly, we have $(hu)^{-1}d(hu) = -\Ad_{hu}^{-1}(p^{-1}dp) + p^{-1}dp$. Thus 
  \[
    (hu)^{-1}d(hu) + d(hu)(hu)^{-1} = -\Ad_{hu}^{-1}(p^{-1}dp) + \Ad_{hu}(p^{-1}dp),
  \]
  whence the desired formula.
\end{proof}

\begin{remark}
  In particular, if $h_0 \in T$ and $Z_G(h_0)=H$, then $\calC = \{ h_0 \}$ and 
  $\tau \colon U^+ \times U^- \to \frkC$ gives an affine chart on $\frkC$.
  Replacing $h_0$ by its Weyl group translates in the definition of $\tau$, 
  we obtain various affine charts which cover $\frkC$.
\end{remark}

\section{Unfolding: Case of Poincar\'e rank $1$}\label{sec:unfolding1}

In the case of $r=1$, the space $\calA(Q)$ is a usual fission space 
$\FS{G}{H}^1 = G \times H \times U^+ \times U^-$ considered in \cite{Boa09}. 
In this section, we relate such a fission space 
to another quasi-Hamiltonian $G \times H$-space introduced in \cite{Boa11}.

As in the previous section, 
let $P$ be a parabolic subgroup of $G$ 
and $H$ be a Levi subgroup of $P$.
Let $U^+$ be the unipotent radical of $P$ with Lie algebra $\frku$ 
and $U^-$ be that of the opposite parabolic subgroup $P^-$.

Let $P^-$ act on the product $G \times P^-$ by $q \cdot (C,p)=(qC,qpq^{-1})$ 
and $\bbM$ be the quotient of $G \times P^-$ by the action of the subgroup $U^-$: 
\[
  \bbM = (G \times P^-)/U^-.
\]
The space $\bbM$ has a residual action of $H = P^-/U^-$ 
and a commuting action of $G$ induced from the action 
$g \cdot (C,p) = (Cg^{-1},p)$ on $G \times P^-$. 
Let $\varpi \colon P^- \to H$ be the canonical projection 
and $\pi \colon G \times P^- \to \bbM$ be the quotient map.

\begin{theorem}[{\cite{Boa11}}]
  The $G \times H$-variety $\bbM$ is a quasi-Hamiltonian $G \times H$-space
  with moment map
  \[
    \mu \colon \bbM \to G \times H, \quad [C,p] \mapsto (C^{-1} p C, \varpi(p)^{-1}),
  \]
  and quasi-Hamiltonian two-form $\omega$ defined by the condition 
  \[
    2\pi^*\omega = (dC\,C^{-1},\Ad_p(dC\,C^{-1})) 
    + (dC\,C^{-1},p^{-1}dp + dp\,p^{-1}).
  \]
\end{theorem}

Take $t \in H$ so that $Z_G(t)^0=H$. 
Note that such an element always exists; the Levi subgroup $H$ is known 
to be the centralizer of some torus $S \subset T$, 
and generic elements of $S$ satisfy the condition (see \prpref{prp:centralizer}).
By \lmmref{lmm:conj-unip}, the set $t U^+$ is contained in the conjugacy class 
$\frkC$ of $t$.
Thus we can define the following map, which we call the \emph{unfolding map} 
for $\FS{G}{H}^1$ associated to $t$:
\[
  \Upsilon_t \colon \FS{G}{H}^1 \to \bbM \circledast_G \frkC, 
  \quad (C,h,u,v) \mapsto ([C,ht^{-1}v],C^{-1}v^{-1}tuvC).
\]
Let $\mu \colon \bbM \circledast_G \frkC \to G \times H$ be the moment map. 

\begin{theorem}\label{thm:unfolding1}
  The map $\Upsilon_t$ is a $G \times H$-equivariant \'etale morphism 
  (an open immersion if $Z_G(t)=H$). 
  It intertwines the quasi-Hamiltonian two-forms and 
  \[
    \mu(\Upsilon_t(C,h,u,v)) = (C^{-1}huvC,t h^{-1}).
  \] 
\end{theorem}

\begin{proof}
  Observe that the map $\Upsilon_t$ is the composite of the two maps
  \begin{align}
    \FS{G}{H}^1 &\to G \times P^- \times U^+, \quad (C,h,u,v) \mapsto (vC,vht^{-1},u); \\
    G \times P^- \times U^+ &\to \bbM \circledast_G \frkC, \quad (C,p,u) \mapsto ([C,p],C^{-1}tuC).
  \end{align}
  The former map is an isomorphism with inverse $(C,p,u) \mapsto (v^{-1}C,kt,u,v)$ where $p=vk$ with $v \in U^-$, $k \in H$. 
  On the other hand, if we let $U^-$ act on $G \times P^- \times U^+ \times U^-$ by 
  \[
    w \cdot (C,p,u,v) = (wC,wpw^{-1},u,vw^{-1}),
  \]
  then the latter map is induced from the $U^-$-equivariant map 
  \[
    \varphi \colon G \times P^- \times U^+ \times U^- \to G \times P^- \times \frkC, \quad
    (C,p,u,v) \mapsto (C,p,C^{-1}v^{-1}tuvC)
  \]
  through the obvious isomorphism $G \times P^- \times U^+ \simeq (G \times P^- \times U^+ \times U^-)/U^-$.
  By \prpref{prp:conj-tri}, $\varphi$ is an \'etale morphism (an open immersion if $Z_G(t)=H$). 
  Thus $\Upsilon_t$ has the same property.
\end{proof}

Let $\bbP$ be the variety of parabolic subgroups of $G$ conjugate to $P^-$. 
For an $H$-conjugacy class $\calC \subset H$, 
define the \emph{enriched conjugacy class} associated to $\calC$ to be 
the variety $\wh{\calC}$ consisting of pairs $(g,P)$ with $P \in \bbP$ and $g \in P$ such that 
$CPC^{-1}=P^-$, $CgC^{-1} \in \calC U^-$ for some $C \in G$. 
As pointed out in \cite{Boa11}, the reduction $\bbM_{\calC} \coloneqq \bbM \spqa{\calC^{-1}} H$ 
is isomorphic to $\wh{\calC}$ via the map 
\[
  \bbM_{\calC} \to \wh{\calC}, \quad [C,p] \mapsto (C^{-1}pC,C^{-1}P^-C),
\]
which induces the structure of a quasi-Hamiltonian $G$-space on $\wh{\calC}$ with moment map $(g,P) \mapsto g$.

\begin{corollary}
  For any $H$-conjugacy class $\calC$, 
  the map $\Upsilon_t$ induces a $G$-equivariant \'etale morphism 
  \[
    \FS{G}{H}^1 \spqa{\calC^{-1}} H \to \wh{\calC}_0 \circledast_G \frkC 
  \]
  intertwining the quasi-Hamiltonian $G$-structures, 
  where $\calC_0 = \calC t^{-1}$.
  It is an open immersion if $Z_G(t)=H$.
\end{corollary}

We are mainly interested in the following situation, 
in which the enriched conjugacy class is a covering of a usual conjugacy class.

\begin{proposition}\label{prp:enrichedconj}
  Take $h_0 \in H$ so that $Z_G(h_0)^0 \subset H$ and let 
  $\calC_0 \subset H$ $($resp.\ $\frkC_0 \subset G)$ be 
  its $H$-conjugacy class $($resp.\ $G$-conjugacy class$)$. 
  Then the moment map $\mu \colon \wh{\calC}_0 \to G$ defines 
  a $G$-equivariant $($surjective$)$ \'etale morphism   
  $\wh{\calC}_0 \to \frkC_0$ intertwining the quasi-Hamiltonian $G$-structures.
  It is an isomorphism if $Z_G(h_0) \subset H$.
\end{proposition}

\begin{proof}
  We identify $\wh{\calC}_0$ with $\bbM_{\calC_0}$. 
  By \lmmref{lmm:conj-unip}, the moment map $\mu \colon [C,p] \mapsto C^{-1}pC$ takes values in $\frkC_0$. 
  Thus $\mu$ defines a morphism $\bbM_{\calC_0} \to \frkC_0$, which is surjective by $G$-equivariance.
  Now observe that the map 
  \[
    \iota \colon \calC_0 \times U^+ \times U^- \to \bbM_{\calC_0}, \quad (h,u,v) \mapsto [u,hv]
  \]
  is an open immersion, along which $\mu \colon \wh{\calC}_0 \to \frkC_0$ is pulled back to  
  \[
    \tau \colon \calC_0 \times U^+ \times U^- \to \frkC_0, \quad (h,u,v) \mapsto u^{-1}h v u,
  \]
  which is \'etale (an open immersion if $Z_G(h_0) = H$) by \prpref{prp:conj-tri}. 
  It is easy to see that the pullback 
  of the quasi-Hamiltonian two-form on $\bbM_{\calC_0} \simeq \bbM \glue{H} \calC_0$ 
  along $\iota$ coincides with the two-form $\tau^*\omega_G$ given in \prpref{prp:conj-tri} with $U^+$ and $U^-$ swapped. 
  Thus $\mu$ intertwines the quasi-Hamiltonian two-forms on an open dense subset of $\wh{\calC}_0$, 
  and hence on the whole $\wh{\calC}_0$ by continuity. 
  Clearly the moment maps for $\wh{\calC}_0$ and $\frkC_0$ match up via $\mu$.
  Hence $\mu \colon \wh{\calC}_0 \to \frkC_0$ is \'etale by the lemma below (in particular, it is quasi-finite). 
  If $Z_G(h_0) \subset H$, then $\mu$ is birational since $\tau$ is an open immersion.
  By Zariski's main theorem (a quasi-finite birational morphism between irreducible varieties is an open immersion if the target variety is normal), it is an isomorphism.
\end{proof}

\begin{lemma}\label{lmm:etale}
  Let $(M,\omega_M,\mu_M)$ and $(N,\omega_N,\mu_N)$ be quasi-Hamiltonian $G$-spaces 
  of the same dimension, 
  and let $\varphi \colon M \to N$ be a $G$-equivariant morphism. 
  If $\varphi^*\omega_N = \omega_M$ and $\mu_N \circ \varphi = \kappa \mu_M$ for some constant central element $\kappa \in Z(G)$, 
  then $\varphi$ is \'etale.
\end{lemma}

\begin{proof}
  Take $p \in M$ and $v \in T_p M$ with $d\varphi(v)=0$. 
  Then for any $w \in T_p M$, we have 
  \[
    (\omega_M)_p(v,w) = (\omega_N)_{\varphi(p)}(d\varphi(v),d\varphi(w)) = 0,
  \]
  i.e., $v \in \Ker (\omega_M)_p$. Furthermore, we have 
  \[
    \mu_M(p)^{-1}d\mu_M(v) = 
    \mu_N(\varphi(p))^{-1}d\mu_N (d\varphi (v)) = 0.
  \]
  Hence $v \in \Ker (\omega_M)_p \cap \Ker d\mu_M$, which is zero by (QH3). 
  Thus $d\varphi$ is injective, and is an isomorphism since $\dim M = \dim N$.
\end{proof}

Thus for $h_0, t \in H$ with $Z_G(h_0) \subset Z_G(t)=H$, 
the map 
\[ 
  \FS{G}{H}^1 \spqa{t^{-1}\calC_0^{-1}} H \to \frkC_0 \circledast_G \frkC, 
  \quad [C,h,u,v] \mapsto (C^{-1}ht^{-1}vC,C^{-1}v^{-1}tuvC)
\]
is a $G$-equivariant open immersion  
intertwining the quasi-Hamiltonian $G$-structures.

\section{Unfolding: General case}\label{sec:unfolding}

In this section we relate the space $\calA(Q)$ associated to an irregular type $Q$ 
with a product of conjugacy classes in $G$.

Consider the multi-fission space $\calA$ 
associated to an increasing sequence of parabolic subgroups 
$P_1 \subset P_2 \subset \cdots \subset P_r$, $r>1$. 
Let $U_j^\pm, H_j$, $j=1,2, \dots ,r$ be as before and 
put 
\[
  \calA' = G \times H_1 \times \prod_{j=1}^{r-1} (U_j^+ \times U_j^-),
\]
which is the multi-fission space associated to 
$P_1 \subset P_2 \subset \cdots \subset P_{r-1}$. 
Take $t \in T$ so that $Z_G(t)^0=H_r$ 
(in particular, $t$ lies in the center of $H_r$), 
and let $\frkC$ be the conjugacy class of $t$.
Then we construct a $G \times H_1$-equivariant morphism  
\[
  \Upsilon \colon \calA \to \calA' \circledast_G \frkC
\]
as follows. 
For $(C,h,u_1, \dots ,u_{2r}) \in \calA$, 
put 
\begin{align}
  k &= ht^{-1}, \\
  v_i &= t u_i t^{-1} \quad (i=1,2, \dots ,2r-3), \\
  v_{2r-2} &= t u_{2r-2} t^{-1} u_{2r}, \\
  M &= C^{-1} u_{2r}^{-1} t u_{2r-1} u_{2r} C.
\end{align}
Then 
\begin{align}
  C^{-1}h u_1 \cdots u_{2r} C 
  &= C^{-1}h t^{-1} v_1 \cdots v_{2r-3} t u_{2r-2} u_{2r-1} u_{2r} C \\
  &= C^{-1}k v_1 \cdots v_{2r-3} v_{2r-2} u_{2r}^{-1} t u_{2r-1} u_{2r} C \\
  &= C^{-1}k v_1 \cdots v_{2r-3} v_{2r-2} C M.
\end{align}
\lmmref{lmm:conj-unip} implies that $M$ lies in $\frkC$, so that 
we obtain a map 
\[
  \Upsilon_t \colon \calA \to \calA' \circledast_G \frkC, \quad 
  (C,h,u_1, \dots ,u_{2r}) \mapsto 
  (C,k,v_1, \dots ,v_{2r-2},M).
\]
Let $\mu \colon \calA' \circledast_G \frkC \to G \times H_1$ be the moment map.

\begin{lemma}\label{lmm:unfolding-induction}
  The map $\Upsilon_t$ is a $G \times H_1$-equivariant \'etale morphism 
  (an open immersion if $Z_G(t)=H_r$). 
  It intertwines the quasi-Hamiltonian two-forms and 
  \[
    \mu(\Upsilon_t(C,h,u_1, \dots ,u_{2r})) = (C^{-1}h u_1 \cdots u_{2r} C, t h^{-1}).
  \]  
\end{lemma}

\begin{proof}
  By \prpref{prp:conj-tri}, the map $\Upsilon_t$ is \'etale, and is an open immersion if $Z_G(t)=H_r$.
  Also, the formula for the moment map is already verified above. 
  Therefore it remains to show that $\Upsilon_t$ intertwines the quasi-Hamiltonian two-forms.
  Let $\omega, \omega'$ be the quasi-Hamiltonian two-forms on $\calA, \calA' \circledast_G \frkC$, respectively 
  and put $\Omega = \Upsilon_t^*\omega' - \omega$. 
  Since the $G$-factors of the moment maps for $\calA$ and $\calA' \circledast_G \frkC$ match up via $\Upsilon_t$, 
  we have 
  \[
    \Omega(v_\xi,\bcdot) = 0, \quad \xi \in \frkg.
  \]
  Hence $\Omega$ descends to a two-form on the quotient 
  \[
    \calA /G \simeq H_1 \times \prod_{j=1}^r (U_j^+ \times U_j^-),
  \]
  and it suffices to show that this two-form is zero, i.e., 
  $\Omega |_{C=1} = 0$. By the definition, we have
  \begin{equation}
    2\omega|_{C=1} = (dC_0\,C_0^{-1},h^{-1}dh) - \sum_{j=0}^{2r-1} (C_j^{-1}dC_j,C_{j+1}^{-1}dC_{j+1}),
  \end{equation}
  where $C_j = u_{j+1} u_{j+2} \cdots u_{2r}$.  
  Put $D_j = v_{j+1} v_{j+2} \cdots v_{2r-2}$ for $j=0,1, \dots ,2r-2$ and $u=u_{2r-1}$, $v=u_{2r}$. 
  Then 
  \[
    C_j = t^{-1}D_j M \ (j=0,1, \dots ,2r-3), \quad C_{2r-2} = uv =t^{-1}vM, \quad C_{2r-1} = v.
  \]
  We have 
  \begin{align}
    (dC_0\,C_0^{-1},h^{-1}dh) 
    &= (\Ad_t^{-1}(dD_0\,D_0^{-1}) + \Ad_{t^{-1}D_0}(dM\,M^{-1}),h^{-1}dh) \\
    &= (dD_0\,D_0^{-1},k^{-1}dk) + (dM\,M^{-1},\Ad_{D_0}^{-1}(k^{-1}dk)).
  \end{align}
  For $j=0,1, \dots ,2r-4$, we have 
  \begin{align}
    (C_j^{-1}dC_j,C_{j+1}^{-1}dC_{j+1})
    &= (M^{-1}dM + \Ad_M^{-1}(D_j^{-1}dD_j), M^{-1}dM + \Ad_M^{-1}(D_{j+1}^{-1}dD_{j+1})) \\ 
    &= (D_j^{-1}dD_j - D_{j+1}^{-1}dD_{j+1}, dM\,M^{-1}) + (D_j^{-1}dD_j,D_{j+1}^{-1}dD_{j+1}).     
  \end{align}
  Noting that $D_{2r-3}^{-1}dD_{2r-3}$ takes values in the Lie algebra of $U_{r-1}^-$, we also have 
  \begin{align}
    (C_{2r-3}^{-1}dC_{2r-3},C_{2r-2}^{-1}dC_{2r-2}) 
    &= (M^{-1}dM + \Ad_M^{-1}(D_{2r-3}^{-1}dD_{2r-3}), M^{-1}dM + \Ad_M^{-1}(v^{-1}dv)) \\
    &= (dM\,M^{-1}, v^{-1}dv) + (D_{2r-3}^{-1}dD_{2r-3}, v^{-1}dv + dM\,M^{-1}) \\
    &= (dM\,M^{-1}, v^{-1}dv) + (D_{2r-3}^{-1}dD_{2r-3}, dM\,M^{-1}),
  \end{align}
  and 
  \[
    (C_{2r-2}^{-1}dC_{2r-2},C_{2r-1}^{-1}dC_{2r-1})
    = (v^{-1}dv + \Ad_v^{-1}(u^{-1}du), v^{-1}dv) = (u^{-1}du, dv\,v^{-1}).
  \]
  Thus we obtain 
  \begin{align}
    2\omega|_{C=1} 
    &= (dD_0\,D_0^{-1},k^{-1}dk) - \sum_{j=0}^{2r-4} (D_j^{-1}dD_j,D_{j+1}^{-1}dD_{j+1}) \\ 
    &\quad + (dM\,M^{-1},\Ad_{D_0}^{-1}(k^{-1}dk)) - (D_0^{-1}dD_0, dM\,M^{-1}) \\
    &\quad - (dM\,M^{-1}, v^{-1}dv) - (u^{-1}du, dv\,v^{-1}).
  \end{align}
  The first line on the right-hand side is equal to $2\omega_{\calA'}$, 
  where $\omega_{\calA'}$ is the quasi-Hamiltonian two-form on $\calA'$ pulled back to 
  $H_1 \times \prod_{j=1}^r (U_j^+ \times U_j^-)$. 
  For the third line, note that $M= v^{-1}tuv$ on $H_1 \times \prod_{j=1}^r (U_j^+ \times U_j^-)$, so that 
  \[
    dM\,M^{-1} = -v^{-1}dv + \Ad_{v^{-1}t}(du\,u^{-1}) + \Ad_M(v^{-1}dv),
  \]
  and hence 
  \[
    - (dM\,M^{-1}, v^{-1}dv) - (u^{-1}du, dv\,v^{-1}) = (dv\,v^{-1}, \Ad_t(du\,u^{-1}) + \Ad_{tu}(dv\,v^{-1}) + u^{-1}du).
  \]
  By \prpref{prp:conj-tri}, this is equal to $2\omega_{\frkC}$, 
  where $\omega_{\frkC}$ is the pullback of the quasi-Hamiltonian two-form on $\frkC$. 
  Thus we obtain 
  \begin{align}
    2\omega|_{C=1} &= 2\omega_{\calA'} + 2\omega_{\frkC} - (\Ad_{D_0}^{-1}(k^{-1}dk) + D_0^{-1}dD_0, dM\,M^{-1}) \\
    &= 2\omega_{\calA'} + 2\omega_{\frkC} - ((kD_0)^{-1}d(kD_0), dM\,M^{-1}) = 2\Upsilon_t^*\omega'|_{C=1}
  \end{align}
  by the definition of the fusion product. 
  Hence the map $\Upsilon_t$ intertwines the quasi-Hamiltonian two-forms.
\end{proof}

If $r>2$, then we can apply \lmmref{lmm:unfolding-induction} to $\calA'$.
Repeating this argument and compositing the resulting morphisms, 
we obtain an \'etale morphism 
\[
  \calA \to \FS{G}{H_1}^1 \circledast_G \frkC_2 \circledast_G \cdots \circledast_G \frkC_r, 
\]
where $\FS{G}{H_1}^1 = G \times H_1 \times U_1^+ \times U_1^-$ and 
each $\frkC_i$ is the $G$-conjugacy class of an element $t_i \in T$ such that $Z_G(t_i)^0=H_i$. 
Finally, if we further take an element $t_1 \in T$ so that 
$Z_G(t_1)^0=H_1$,
then by \thmref{thm:unfolding1}, we obtain an \'etale morphism $\FS{G}{H_1}^1 \to \bbM \circledast_G \frkC_1$, 
where $\bbM$ is the quasi-Hamiltonian $G \times H_1$-space introduced in the previous section and 
$\frkC_1$ is the conjugacy class of $t_1$.
Composing these morphisms, we obtain the following theorem.

\begin{theorem}\label{thm:unfolding}
  % Let $\calA$ be the multi-fission space associated to 
  % an increasing sequence of parabolic subgroups
  % $P_1 \subset P_2 \subset \cdots \subset P_r$, and  
  % Let $U_j^\pm, H_j$, $j=1,2, \dots ,r$ be as before.
  For each $i=1,2, \dots ,r$, take $t_i \in T$ so that $Z_G(t_i)^0=H_i$
  and let $\frkC_i$ be its $G$-conjugacy class. 
  Also, let $\bbM = (G \times H_1 U_1^-)/U_1^-$ be the quasi-Hamiltonian $G \times H_1$-space introduced in the previous section. 
  Then there exists a(n explicit) $G \times H_1$-equivariant \'etale morphism
  \[
    \Upsilon_{t_1,\dots ,t_r} \colon \calA \to \bbM \circledast_G \frkC_1 \circledast_G \cdots \circledast_G \frkC_r, 
  \]
  satisfying the following conditions:
  \begin{enumerate}
    \item it is an open immersion if $Z_G(t_i)=H_i$ for all $i=1,2, \dots ,r$;
    \item it intertwines the quasi-Hamiltonian two-forms;
    \item the pullback of the moment map $\mu \colon \bbM \circledast_G \frkC_1 \circledast_G \cdots \circledast_G \frkC_r \to G \times H_1$ is given by  
    \[
      \mu(\Upsilon_{t_1,\dots ,t_r}(C,h,u_1,u_2, \dots ,u_{2r})) = (C^{-1}hu_1 u_2 \cdots u_{2r}C,t_1 t_2 \cdots t_r h^{-1}).  
    \]
  \end{enumerate}
\end{theorem}

We call $\Upsilon_{t_1,\dots ,t_r}$ 
the \emph{unfolding map} for $\calA$ associated to $(t_1,\dots ,t_r)$. 
It is explicitly described as follows:
For $(C,h,u_1^+,u_1^-, \dots ,u_r^+,u_r^-) \in \calA$  
(so that $u_j^+ \in U_j^+$, $u_j^- \in U_j^-$), 
put 
\[
  v_i^\pm = (t_{i+1}t_{i+2} \cdots t_r) u_i^\pm (t_{i+1}t_{i+2} \cdots t_r)^{-1} 
  \quad (i=1,2, \dots ,r),
\]
and define 
\[
  M_i = C^{-1}(v_i^- v_{i+1}^- \cdots v_r^-)^{-1} t_i v_i^+ (v_i^- v_{i+1}^- \cdots v_r^-)C \quad (i=1,2, \dots ,r).
\]
Then 
\[
  \Upsilon_{t_1,\dots ,t_r}(C,h,u_1^+,u_1^-, \dots ,u_r^+,u_r^-) 
  = ([C,h t_r^{-1} \cdots t_1^{-1} v_1^- v_2^- \cdots v_r^-],M_1,M_2, \dots ,M_r).
\]

Together with \prpref{prp:enrichedconj}, \thmref{thm:unfolding} implies the following.

\begin{corollary}\label{crl:unfolding}
  Let $\calC$ be an $H_1$-conjugacy class and  
  take $t_1, t_2, \dots ,t_r \in T$ satisfying the following conditions: 
  \begin{enumerate}
    \item $Z_G(t_i)=H_i$ for all $i=1,2, \dots ,r$;
    \item $Z_G(h (t_1 t_2 \cdots t_r)^{-1}) \subset H_1 \quad (h \in \calC)$.
  \end{enumerate}
  Let $\frkC_0$ be the $G$-conjugacy class containing  
  $\calC_0 \coloneqq \calC t_r^{-1} \cdots t_1^{-1}$ 
  and for $i=1,2, \dots ,r$, let $\frkC_i$ be the $G$-conjugacy class of $t_i$.
  Then the unfolding map $\Upsilon_{t_1,\dots ,t_r}$ induces a $G$-equivariant 
  morphism   
  \[
    \calA \spqa{\ov{\calC}^{-1}} H_1 \to \ov{\frkC_0} \circledast_G \frkC_1 \circledast_G \cdots \circledast_G \frkC_r
  \]
  intertwining the Hamiltonian quasi-Poisson $G$-structures. 
  Furthermore it restricts to an open immersion 
  $\calA \spqa{\calC^{-1}} H_1 \to \frkC_0 \circledast_G \frkC_1 \circledast_G \cdots \circledast_G \frkC_r$ 
  intertwining the quasi-Hamiltonian $G$-structures. 
\end{corollary}

For any $H_1$-conjugacy class $\calC$, 
there exist $t_1, t_2, \dots ,t_r \in T$ satisfying the above two conditions; 
see \appref{app:parameter}.

\begin{remark}\label{rmk:unfolding-additive}
  We intend to apply the above result in the following situation.
  Take an irregular type  
  $Q(z)=\sum_{j=1}^r Q_j z^{-j} \in z^{-1}\frkt[z^{-1}]$ 
  of pole order $r$ and  
  a non-resonant\footnote{$X \in \frkh$ is said to be non-resonant 
  if $\ad_X \in \End_\C(\frkh)$ does not have nonzero integral eigenvalues; 
  any non-resonant element $X$ satisfies $Z_G(X)=Z_G(\exp(2\pi \sqrt{-1}X))^0$.} 
  element $\Lambda_0$ of $\frkh = \Lie Z_G(Q)$. 
  Then the associated \emph{normal form} is 
  \[
  \Lambda = \sum_{j=0}^r \Lambda_j z^{-j-1}dz = dQ + \Lambda_0 z^{-1}dz.
  \]
  Take $\epsilon = (\epsilon_0,\epsilon_1, \dots ,\epsilon_r) \in \C^{r+1}$ 
  with $\epsilon_i \neq \epsilon_j$ ($i \neq j$). 
  Following \cite{Hir24}, define the \emph{unfolding} of $\Lambda$ to be 
  \[
  \wh{\Lambda} = \sum_{j=0}^r \frac{\Lambda_j\,dz}{(z-\epsilon_0)(z-\epsilon_1) \cdots (z-\epsilon_j)},
  \] 
  which is a $\frkh$-valued logarithmic one-form on the affine line $\C$ with poles on $\{ \epsilon_0, \dots ,\epsilon_r \}$.
  Put  
  \[
  \wh{\Lambda}_i = \res_{z=\epsilon_i} \wh{\Lambda} 
  = \sum_{j=i}^r \Lambda_j \prod_{\sumfrac{0 \leq l \leq j}{l \neq i}} (\epsilon_i - \epsilon_l)^{-1}.
  \]
  Observe that each $\wh{\Lambda}_i$ is a linear combination of $\Lambda_i, \Lambda_{i+1}, \dots , \Lambda_r$ 
  and 
  \[
  \sum_{i=0}^r \wh{\Lambda}_i = -\res_{z=\infty} \wh{\Lambda} = \Lambda_0.
  \] 
  We assume that $\epsilon$ satisfies the following conditions:
  \begin{enumerate}
    \item $Z_G(e^{2\pi\sqrt{-1}\wh{\Lambda}_i})=Z_G(\Lambda_i,\Lambda_{i+1}, \dots ,\Lambda_r)$ ($i=1,2, \dots ,r$);
    \item $Z_G(e^{2\pi\sqrt{-1}\wh{\Lambda}_0}) \subset Z_G(Q)=Z_G(\Lambda_1,\dots ,\Lambda_r)$. %where $(\wh{\Lambda}_0)_s$ (resp.\ $(\Lambda_0)_s$) is the semisimple part of $\wh{\Lambda}_0$ (resp.\ $\Lambda_0$);
    % \item $\wh{\Lambda}_i$ is non-resonant in $\frkg$ for any $i=0,1, \dots ,r$.
  \end{enumerate}
  Let $\calC$ be the $Z_G(Q)$-conjugacy class of $e^{2\pi\sqrt{-1}\wh{\Lambda}_0}$ 
  and put $t_i = e^{2\pi\sqrt{-1}\wh{\Lambda}_i}$ for $i=1,2, \dots ,r$. 
  Then we may apply the above corollary to $\calA(Q)$ and $\calC$, $(t_1,t_2, \dots ,t_r)$. 
\end{remark}

\begin{remark}
  There is an additive analogue of the unfolding map; see \cite{CRY25}.
\end{remark}

\section{Unfolding of wild character varieties}\label{sec:unfolding-wcv}

Let $\bSig = (\Sigma,\bfa;{}^1 Q, \dots ,{}^m Q)$ be an irregular curve of genus $g$.  
As in \secref{sec:wcv}, each irregular type ${}^i Q = \sum_{j=1}^{r_i} {}^i Q_j z_i^{-j}$ (where $r_i$ is the pole order of ${}^i Q$)
defines an increasing sequence of reductive subgroups 
\[
  Z_G({}^i Q) = {}^i H_1 \subset {}^i H_2 \subset \cdots \subset {}^i H_{r_i}, 
  \quad {}^i H_j = Z_G({}^i Q_j, {}^i Q_{j+1}, \dots ,{}^i Q_{r_i}).
\]
By \prpref{prp:triangular}, for each $i$ there exists an increasing sequence of parabolic subgroups
\[
  {}^i P_1 \subset {}^i P_2 \subset \cdots \subset {}^i P_{r_i} 
\]
with each ${}^i P_j$ containing ${}^i H_j$ as a Levi subgroup such that 
$\calA({}^i Q)$ is isomorphic to the multi-fission space 
\[
  {}^i \calA = G \times Z_G({}^i Q) \times \prod_{j=1}^{r_i} ({}^i U_j^+ \times {}^i U_j^-),
\]
where ${}^i U_j^\pm$ are the unipotent radicals of ${}^i P_j$ and its opposite parabolic subgroup.
Recall that 
for a conjugacy class $\calC = \prod_{i=1}^m \calC_i$ in $\bfH = \prod_{i=1}^m Z_G({}^i Q)$,
the wild character variety $\MB^\st(\bSig,\calC)$ is 
the reduction of the stable locus of 
$\tMB(\bSig,\calC) = \bbD(G)^{\circledast g} \circledast_G \calA_{\calC_1}({}^1 Q) \circledast_G \cdots \circledast_G \calA_{\calC_m}({}^m Q)$
by $G'$.

By \prpref{prp:enrichedconj} and \crlref{crl:unfolding}, 
we obtain the following theorem.

\begin{theorem}\label{thm:unfolding-wcv}
  Let $\calC = \prod_{i=1}^m \calC_i$ be a conjugacy class in $\bfH$.
  For each $i=1,2, \dots ,m$, 
  take ${}^i t_j \in T$, $j=1,2, \dots ,r_i$ satisfying  
  \[
    Z_G({}^i t_j) = {}^i H_j \ (j=1,2,\dots,r_i), \quad 
    Z_G(h ({}^i t_1 {}^i t_2 \cdots {}^i t_{r_i})^{-1}) \subset Z_G({}^i Q) \ (h \in \calC_i),
  \]
  and let ${}^i \frkC_j$ be the $G$-conjugacy class of ${}^i t_j$. 
  Also let ${}^i \frkC_0$ be the $G$-conjugacy class containing  
  $\calC_i ({}^i t_1 {}^i t_2 \cdots {}^i t_{r_i})^{-1}$, 
  and put 
  \[
    {}^i \frkC = {}^i \frkC_0 \circledast_G {}^i \frkC_1 \circledast_G \cdots \circledast_G {}^i \frkC_{r_i} \subset G^{r_i+1}.
  \]
  Then the unfolding maps $\Upsilon_{{}^i t_1, \dots ,{}^i t_{r_i}}$, 
  $i=1,2, \dots ,m$ give rise to a $G$-equivariant morphism 
  \[
    \Upsilon_{({}^i t_j)} \colon \tMB(\bSig,\ov{\calC}) \to \bbD(G)^{\circledast g} \circledast_G \ov{{}^1 \frkC} \circledast_G \cdots \circledast_G \ov{{}^m \frkC}
  \] 
  intertwining the Hamiltonian quasi-Poisson $G$-structures. Furthermore it restricts to an open immersion 
  $\tMB(\bSig,\calC) \to \bbD(G)^{\circledast g} \circledast_G {}^1 \frkC \circledast_G \cdots \circledast_G {}^m \frkC$ 
  intertwining the quasi-Hamiltonian $G$-structures. 
\end{theorem}

Take $\sum_{i=1}^m (r_i+1)$ marked points 
$\bfb = \{\, {}^i b_j \mid i=1,2,\dots,m, j=0,1,\dots,r_i\,\}$ on $\Sigma$ 
and consider the irregular curve $\bSig' = (\Sigma,\bfb;\mathbf{0})$ 
with trivial irregular type $0$ at each marked point. Then 
\[
  \bbD(G)^{\circledast g} \circledast_G \ov{{}^1 \frkC} \circledast_G \cdots \circledast_G \ov{{}^m \frkC} 
  = \tMB(\bSig',\ov{\frkC}), \quad 
  \frkC = \prod_{i=1}^m \prod_{j=0}^{r_i} {}^i \frkC_j,
\]
and the above map $\Upsilon_{({}^i t_j)}$ induces a Poisson morphism 
\[
  \Upsilon_{({}^i t_j)} \colon \MB(\bSig,\ov{\calC}) \to \MB(\bSig',\ov{\frkC}).
\] 
It restricts to an open immersion 
\[
  \MB^\circ(\bSig,\calC) \to \MB^\st(\bSig',\frkC) 
\]
intertwining the symplectic structures, where 
\[
  \MB^\circ(\bSig,\calC) 
  = \{\, p \in \tMB(\bSig,\calC) \mid \text{$\Upsilon_{({}^i t_j)}(p)$ is a stable point in $\tMB(\bSig',\frkC)$} \,\} \spq G,
\]
which is an open subset of $\MB^\st(\bSig,\calC)$.
In particular, the following holds. 
\begin{corollary}\label{corl:unfolding-wcv}
  In the situation of \thmref{thm:unfolding-wcv}, 
  if $\MB^\circ(\bSig,\calC)$ is nonempty, 
  then $\MB(\bSig,\ov{\calC})$ and $\MB(\bSig',\ov{\frkC})$ 
  have Poisson birationally equivalent irreducible components.
  Namely, there exist nonempty open subsets of 
  $\MB(\bSig,\ov{\calC})$ and $\MB(\bSig',\ov{\frkC})$ 
  which are isomorphic as Poisson varieties.
\end{corollary}

\thmref{thm:main} follows from this corollary; 
the assumption of \thmref{thm:main} implies that $\MB^\circ(\bSig,\calC)$ is nonempty. 

\appendix 

\section{The space of unfolding parameters}\label{app:parameter}

In this appendix, we go back to the situation of \secref{sec:unfolding} and 
show that for any $H_1$-conjugacy class $\calC$, 
there exist $t_1, t_2, \dots ,t_r \in T$ satisfying the conditions in \crlref{crl:unfolding}.

\begin{proposition}\label{prp:centralizer}
  Let $S$ be a torus in $G$. 
  Then there exist characters $\chi_1, \chi_2, \dots ,\chi_l \colon S \to \C^\times$ 
  such that $Z_G(g)=Z_G(S)$ for any $g \in S$ with $\chi_i(g) \neq \chi_j(g)$ ($i \neq j$).
\end{proposition}

\begin{proof}
  See \cite[Lemma~6.4.3]{Spr} and its proof.
\end{proof}

\begin{proposition}\label{prp:parameter}
  Let $H$ be a Levi subgroup of a parabolic subgroup of $G$ 
  containing $T$, and let $Z \subset T$ be a torus with $Z_G(Z)=H$ 
  $($for instance, take $Z$ to be the identity component of the center of $H)$.  
  Then for any $H$-conjugacy class $\calC \subset H$, 
  the set of $t \in Z$ satisfying the following conditions contains an open dense subset of $Z$: 
  \begin{enumerate}
    \item $Z_G(t)=H$;
    \item $Z_G(h t^{-1}) \subset H \ (h \in \calC)$.
  \end{enumerate}  
\end{proposition}

\begin{proof}
  Take $h \in \calC$ so that its semisimple part $h_s$ lies in $T$, 
  and let $S \subset G$ be the smallest torus containing $h_s Z$. 
  Then $h_s \in S$ and hence $Z = h_s^{-1} \cdot h_S Z \subset S$. 
  Thus $H=Z_G(Z) \supset Z_G(S)$. 
  Take characters $\chi_1, \dots ,\chi_l \colon S \to \C^\times$ as in \prpref{prp:centralizer} 
  and put 
  \[
    S' = \{\, t \in S \mid \chi_i(t) \neq \chi_j(t) \ (i \neq j) \,\}.
  \]
  If $S' \cap h_s Z = \emptyset$, then $\chi_i = \chi_j$ on $h_s Z$ for some $i \neq j$, 
  and hence $h_s Z$ is contained in the identity component of the kernel of $\chi_i \chi_j^{-1}$, 
  which contradicts the minimality of $S$. 
  Hence $S' \cap h_s Z \neq \emptyset$. 
  Since the set of $t \in Z$ such that $Z_G(t)=Z_G(Z)(=H)$ contains a nonempty open subset, 
  the result follows.
\end{proof}

Thus if we take a torus $Z_j \subset T$ so that $Z_G(Z_j)=H_j$ for each $j=1,2, \dots ,r$,
then generic elements $(t_1,t_2, \dots ,t_r)$ of $Z_1 \times Z_2 \times \cdots \times Z_r$
satisfy the conditions in \crlref{crl:unfolding}.

\bibliographystyle{amsplain}
\bibliography{refs-ymkw}

\providecommand{\bysame}{\leavevmode\hbox to3em{\hrulefill}\thinspace}
\providecommand{\MR}{\relax\ifhmode\unskip\space\fi MR }
\providecommand{\MRhref}[2]{%
  \href{http://www.ams.org/mathscinet-getitem?mr=#1}{#2}
}
\providecommand{\href}[2]{#2}
\begin{thebibliography}{10}

\bibitem{AKSM02}
A.~Alekseev, Y.~Kosmann-Schwarzbach, and E.~Meinrenken, \emph{Quasi-{P}oisson manifolds}, Canad. J. Math. \textbf{54} (2002), no.~1, 3--29. \MR{1880957}

\bibitem{AMM98}
A.~Alekseev, A.~Malkin, and E.~Meinrenken, \emph{Lie group valued moment maps}, J. Differential Geom. \textbf{48} (1998), no.~3, 445--495. \MR{1638045}

\bibitem{Boa09}
P.~Boalch, \emph{Quivers and difference {P}ainlev\'e equations}, Groups and symmetries, CRM Proc. Lecture Notes, vol.~47, Amer. Math. Soc., Providence, RI, 2009, pp.~25--51. \MR{MR2500553}

\bibitem{Boa11}
\bysame, \emph{Riemann-{H}ilbert for tame complex parahoric connections}, Transform. Groups \textbf{16} (2011), no.~1, 27--50. \MR{2785493}

\bibitem{Boa14}
\bysame, \emph{Geometry and braiding of {S}tokes data; fission and wild character varieties}, Ann. of Math. (2) \textbf{179} (2014), no.~1, 301--365. \MR{3126570}

\bibitem{BY15}
P.~Boalch and D.~Yamakawa, \emph{Twisted wild character varieties}, arXiv:1512.08091.

\bibitem{Bor}
A.~Borel, \emph{Linear algebraic groups}, second ed., Graduate Texts in Mathematics, vol. 126, Springer-Verlag, New York, 1991. \MR{1102012}

\bibitem{BC05}
H.~Bursztyn and M.~Crainic, \emph{Dirac structures, momentum maps, and quasi-{P}oisson manifolds}, The breadth of symplectic and {P}oisson geometry, Progr. Math., vol. 232, Birkh\"{a}user Boston, Boston, MA, 2005, pp.~1--40. \MR{2103001}

\bibitem{BCS05}
H.~Bursztyn, M.~Crainic, and P.~\v{S}evera, \emph{Quasi-{P}oisson structures as {D}irac structures}, Travaux math\'{e}matiques. {F}asc. {XVI}, Trav. Math., vol.~16, Univ. Luxemb., Luxembourg, 2005, pp.~41--52. \MR{2223150}

\bibitem{CRY25}
M.~Chaffe, G.~Rembado, and D.~Yamakawa, \emph{Wild genus-zero quantum de rham spaces}, arXiv:2510.17666.

\bibitem{Hir24}
K.~Hiroe, \emph{Deformation of moduli spaces of meromorphic {$G$}-connections on $\mathbb{P}^{1}$ via unfolding of irregular singularities}, arXiv:2407.20486.

\bibitem{Kli}
M.~Klime\v{s}, \emph{Wild monodromy of the {F}ifth {P}ainlev\'e{} equation and its action on wild character variety: an approach of confluence}, Ann. Inst. Fourier (Grenoble) \textbf{74} (2024), no.~1, 121--192. \MR{4748169}

\bibitem{LBS15}
D.~Li-Bland and P.~\v{S}evera, \emph{Moduli spaces for quilted surfaces and {P}oisson structures}, Doc. Math. \textbf{20} (2015), 1071--1135. \MR{3424475}

\bibitem{PauRam}
E.~Paul and J.-P. Ramis, \emph{Dynamics of the fifth {P}ainlev\'e{} foliation}, Handbook of geometry and topology of singularities {VI}. {F}oliations, Springer, Cham, [2024] \copyright 2024, pp.~307--381. \MR{4789491}

\bibitem{Ses72}
C.~S. Seshadri, \emph{Quotient spaces modulo reductive algebraic groups}, Ann. of Math. (2) \textbf{95} (1972), 511--556; errata, ibid. (2) 96 (1972), 599. \MR{309940}

\bibitem{Spr}
T.~A. Springer, \emph{Linear algebraic groups}, second ed., Modern Birkh\"auser Classics, Birkh\"auser Boston, Inc., Boston, MA, 2009. \MR{2458469}

\end{thebibliography}

\end{document}